\documentclass[11pt]{article}
\usepackage{amsmath,amsthm,amssymb,amsfonts}
\usepackage{mathrsfs}
\usepackage{cite}

\newtheorem{theorem}{Theorem}[section]

\newtheorem{lemma}[theorem]{Lemma}
\newtheorem{corollary}[theorem]{Corollary}
\newtheorem{remark}{Remark}
\newtheorem{problem}{Problem}
\newtheorem{conjecture}{Conjecture}[section]

\title{Spanning rigid subgraph packing and  sparse subgraph covering}
\author{Xiaofeng Gu\thanks{Department of Mathematics,
University of West Georgia, Carrollton, GA 30118, USA; Email: xgu@westga.edu;
Research partially supported by a grant from the Simons Foundation (522728, XG)}
}

\begin{document}
\date{}
\maketitle
\noindent
\begin{abstract}
Rigidity, arising in discrete geometry, is the property of a structure that does not flex.
Laman provides a combinatorial characterization of rigid graphs in the Euclidean plane,
and thus rigid graphs in the Euclidean plane have applications in graph theory.
We discover a sufficient partition condition of packing spanning rigid subgraphs and spanning trees.
As a corollary, we show that a simple graph $G$ contains a packing of $k$ spanning rigid subgraphs 
and $l$ spanning trees if $G$ is $(4k+2l)$-edge-connected, and $G-Z$ is essentially 
$(6k+2l - 2k|Z|)$-edge-connected for every $Z\subset V(G)$.
Thus every $(4k+2l)$-connected and essentially $(6k+2l)$-connected graph $G$
contains a packing of $k$ spanning rigid subgraphs and $l$ spanning trees.
Utilizing this, we show that every $6$-connected and essentially $8$-connected 
graph $G$ contains a spanning tree $T$ such that $G-E(T)$ is $2$-connected.
These improve some previous results. 
Sparse subgraph covering problems are also studied.

\end{abstract}

\section{Introduction}
Rigidity, arising in discrete geometry, is the property of a structure that does not flex. 
A {\bf $d$-dimensional framework} is a pair $(G, p)$, where $G(V, E)$ is a graph and
$p$ is a map from $V$ to $\mathbb{R}^d$. Roughly speaking, it is a straight line realization
of $G$ in $\mathbb{R}^d$. Two frameworks $(G, p)$ and $(G, q)$ are {\bf equivalent} if
$||p(u) - p(v) || = ||q(u) - q(v) ||$ holds for every edge $uv\in E$, where $||\cdot ||$
denotes the Euclidean norm in $\mathbb{R}^d$. Two frameworks $(G, p)$ and $(G, q)$ 
are {\bf congruent} if $||p(u) - p(v) || = ||q(u) - q(v) ||$ holds for every pair $u, v\in V$.
A framework $(G,p)$ is {\bf generic} 
if the coordinates of all the points are algebraically independent over the rationals. 
The framework $(G, p)$ is {\bf rigid} if there exists an $\varepsilon >0$ such that
if $(G, p)$ is equivalent to $(G, q)$ and $||p(u) - q(u) || < \varepsilon$ for every $v\in V$,
then $(G, p)$ is congruent to $(G, q)$. 
For more about rigidity of $d$-dimensional framework, see \cite{Whit96}.
A generic realization of $G$ is rigid in $\mathbb{R}^d$ if and only if every generic 
realization of $G$ is rigid in $\mathbb{R}^d$. 
Hence the generic rigidity can be considered as a property of the underlying graph.
A graph is {\bf rigid} in $\mathbb{R}^d$ if every generic realization of $G$ is 
rigid in $\mathbb{R}^d$ \cite{AsRo78}. Laman~\cite{Lama70} provides 
a combinatorial characterization of rigid graphs in $\mathbb{R}^2$.  By Laman~\cite{Lama70}, 
the equivalent definition of rigid graphs in $\mathbb{R}^2$ is given below.

Let $G(V, E)$ be a graph. 
For a subset $X\subseteq V(G)$, $G[X]$ denotes the subgraph of $G$ induced by $X$.
For a subset $F\subseteq E(G)$, $G[F]$ is the subgraph of $G$ induced by $F$,
while $G(F)$ denotes the spanning subgraph of $G$ with edge set $F$.
For any partition $\pi$ of $V(G)$, $e_G (\pi)$ denotes
the number of edges of $G$ whose ends lie in two different parts of $\pi$.
A part of $\pi$ is {\bf trivial} if the part consists of a single vertex.
A partition is {\bf nontrivial} if it contains no trivial parts.
Undefined graph terminologies can be found in \cite{BoMu08}.

A {\bf packing} of $l$ spanning subgraphs in a graph $G$ means that $G$ has $l$ edge-disjoint 
spanning subgraphs. The following theorem of Nash-Williams and Tutte characterizes graphs with
a packing of $l$ spanning trees. A sufficient edge connectivity condition of spanning
tree packing can be easily obtained from the theorem.

\begin{theorem}[Nash-Williams~\cite{Nash61} and Tutte~\cite{Tutte61}]
\label{NaTu}
Let $l\ge 0$ be an integer.
A graph $G$ has $l$ edge-disjoint spanning trees if and only if
for any partition $\pi$ of $V(G)$, $e_G (\pi)\ge l(|\pi|-1)$.
\end{theorem}
\begin{corollary}
Every $2l$-edge-connected graph contains $l$ edge-disjoint spanning trees.
\end{corollary}

The packing of spanning trees was then extended to packing of spanning rigid subgraphs 
by \cite{Jord05}.

For a subset $X\subseteq V(G)$, let $i_G(X)$ (or simply $i(X)$ if $G$ can be 
understood from the context) denote the number of edges in $G[X]$.
A graph $G$ is {\bf sparse} if $i_G (X)\le 2|X|-3$ for every $X\subseteq V(G)$
with $|X|\ge 2$. If in addition $|E(G)|=2|V(G)|-3$, then $G$ is {\bf minimally rigid}.
A minimally rigid graph is also called a {\bf Laman graph}. By definition, any sparse graph is simple.
A graph $G$ is {\bf rigid} if $G$ contains a spanning minimally rigid subgraph.
More about rigid graphs will be given in the next section, or see \cite{Gu17, LoYe82, Jord05, JaJo09, ChGS14}.
It is not hard to see that every rigid graph with at least 3 vertices is $2$-connected.
However, the sufficient condition for rigidity is much harder, even in $\mathbb{R}^2$.

\begin{theorem}[Lov\'asz and Yemini \cite{LoYe82}]
\label{LYthm}
Every $6$-connected graph is rigid.
\end{theorem}

\begin{theorem}[Jord\'an \cite{Jord05}]
\label{Jrigi}
Every $6k$-connected graph contains $k$ edge-disjoint spanning rigid subgraphs.
\end{theorem}

\begin{theorem}[Jackson and Jord\'an \cite{JaJo09}]
A simple graph $G$ is rigid if $G$ is $6$-edge-connected, $G-v$ is $4$-edge-connected 
for every $v\in V(G)$ and $G-\{u, v\}$ is $2$-edge-connected for every $u,v\in V(G)$.
\end{theorem}

\begin{theorem}[Cheriyan, Durand de Gevigney and Szigeti \cite{ChGS14}]
\label{CDSthm}
Let $k\ge 1$ and $l\ge 0$ be integers.
A simple graph $G$ contains edge-disjoint $k$ spanning rigid subgraphs and $l$ spanning trees if
$G-Z$ is $(6k+2l - 2k|Z|)$-edge-connected for every $Z\subset V(G)$.
\end{theorem}

Theorem~\ref{NaTu} suggests that the partition condition is ``tighter'' than the
connectivity condition.
This motivates us to find a ``tighter'' condition for packing rigid subgraphs. 
In Theorem~\ref{parthm},  we discover a sufficient partition condition for packing
spanning rigid subgraphs and spanning trees. We utilize it and obtain a sufficient condition 
involving edge connectivity and essential edge connectivity in Corollary~\ref{parcor},
which is an improvement of Theorem~\ref{CDSthm}. 
Some other neater corollaries are also obtained.

 Let $Z\subset V(G)$ and $\pi$ be a partition of $V(G-Z)$ with $n_0$ 
 trivial parts $v_1, v_2,\cdots,v_{n_0}$.  We define $n_Z(\pi)$ to be
$\sum_{1\le i\le n_0} |Z_i|$ where $Z_i$ is the set of vertices in $Z$ that are adjacent
to $v_i$ for $1\le i\le n_0$. If $Z=\emptyset$, then define $n_Z(\pi) =0$.

\begin{theorem}
\label{parthm}
Let $k\ge 1$ and $l\ge 0$ be integers. 
A simple graph $G$ contains edge-disjoint $k$ spanning rigid subgraphs and $l$ spanning trees if
for  any partition $\pi$ of $V(G-Z)$ with $n_0$ trivial parts, 
$e_{G-Z} (\pi)\ge (3k+l)(|\pi|-1) - k n_0 - k|Z|(|\pi|-n_0) - n_Z(\pi)$, for every $Z\subset V(G)$.
\end{theorem}

We must point out that our proof of Theorem~\ref{parthm} is a standard matroid proof using
rank functions, which is similar to \cite{ChGS14, Jord05, Gu17}. We use graph properties and 
a more careful counting to simplify the proof and improve the bounds.

A graph $G$ is {\bf essentially $p$-edge-connected} if $d_G(U)\ge p$ for every partition 
$(U,  V\backslash U)$ of $V(G)$ with $2\le |U| \le |V(G)|-2$ and each part inducing a subgraph 
with at least one edge,  where $d_G(U)$ is the number of  edges between $U$ and 
$V\backslash U$ in $G$.
A graph $G$ with at least $p+1$ vertices is {\bf essentially $p$-connected} if there is no 
$X\subset V(G)$ with $|X|< p$ such that at least two components of $V-X$ are nontrivial,
where a nontrivial component means it contains at least one edge.

\begin{corollary}
\label{parcor}
Let $k\ge 1$ and $l\ge 0$ be integers. 
A simple graph $G$ contains edge-disjoint $k$ spanning rigid subgraphs and $l$ spanning trees if
$G$ is $(4k+2l)$-edge-connected, and
$G-Z$ is essentially $(6k+2l - 2k|Z|)$-edge-connected for every $Z\subset V(G)$.
\end{corollary}

\begin{corollary}
If a simple graph $G$ is $4k$-edge-connected and essentially $6k$-edge-connected,
$G-v$ is essentially $4k$-edge-connected for every $v\in G$, and $G-u-v$ is essentially
$2k$-edge-connected for every $u, v\in G$, then $G$
contains $k$ edge-disjoint spanning rigid subgraphs.
\end{corollary}

A sufficient condition involving ``essentially 6-connected'' for rigidity is shown in \cite{JaSS07}, 
as an improvement of Theorem~\ref{LYthm}. We extends this result to a packing of spanning
rigid subgraphs and spanning trees in the following corollary.

\begin{corollary}
\label{esscon}
Every $(4k+2l)$-connected and essentially $(6k+2l)$-connected graph $G$
contains edge-disjoint $k$ spanning rigid subgraphs and $l$ spanning trees.
\end{corollary}

Kriesell conjectures that there exists a (smallest) integer $f(p)$ such that every $f(p)$-connected
graph $G$ has a spanning tree $T$ such that $G-E(T)$ is $p$-connected (see \cite{Jord05}).
By Theorem~\ref{Jrigi}, Jord\'an \cite{Jord05} shows that every 12-connected graph $G$ contain
a spanning tree $T$ such that $G-E(T)$ is $2$-connected. This is  improved by
Cheriyan, Durand de Gevigney and Szigeti \cite{ChGS14} who show that every 8-connected 
graph $G$ contain a spanning tree $T$ such that $G-E(T)$ is $2$-connected.
By Corollary~\ref{esscon}, we have the following result.
\begin{corollary}
Every $6$-connected and essentially $8$-connected graph $G$
contains a spanning tree $T$ such that $G-E(T)$ is $2$-connected.
\end{corollary}

Necessary conditions for packing rigid subgraphs are also investigated. This appears in Section 4.
In Section 5, we study the sparse subgraph covering problems, including a NDT-type theorem 
and an open question.

We would like to mention that Theorems~\ref{parthm} and \ref{rigidcover} have 
important applications in spectral graph theory. 
Utilizing them, in a subsequent paper, we will study spectrum of graphs
and rigidity in the Euclidean plane, and discover the spectral conditions for 
rigid subgraph packing and sparse subgraph covering.

\section{Preliminaries}
In this section, we present some basic results on rigid graphs and rigidity matroids.
For readers who are interested in more about the relationship between the studies of
rigidity and matroids, we refer to the monograph \cite{GrSS93}.

Suppose that $G=(V,E)$ is a graph with $|V(G)|=n$. Let $\mathcal{F}$ 
be the collection of all edge subsets each of which induces a forest.
Then $\mathcal{F}$ forms the collection of independent sets of a matroid on ground set $E$.
The {\bf circuit matroid} $\mathcal{M}(G)$ of $G$ is the matroid $(E,\mathcal{F})$.
The rank function of $\mathcal{M}(G)$ is given by $r_{\mathcal M}(F) = n - c(F)$,
where $c(F)$ denotes the number of components of $G(F)$.

For any subset $X\subseteq V$ and $F\subseteq E$, $E_F(X)$ and $i_F(X)$ denotes the set
and the number of edges of $F$ in $G[X]$, respectively. 
A subset $S\subseteq E$ is {\bf sparse} if $i_S(X)\le 2|X|-3$ for all
$X\subseteq V$ with $|X|\ge 2$. Let $\mathcal{S}$ be the collection of all sparse sets of $G$.
Then $\mathcal{S}$ forms the collection of independent sets of a matroid on ground set $E$.
The matroid $(E, \mathcal{S})$ is the {\bf rigidity matroid} of $G$, denoted by $\mathcal{R}(G)$.
By Lov\'asz and Yemini \cite{LoYe82}, the rank function of $\mathcal{R}(G)$ is

\begin{equation}
\label{rankrm}
r_{\mathcal R}(F)=\min\left\{\sum_{X\in\mathcal{G}} (2|X|-3) \right\},
\end{equation}
where the minimum is taken over all collections $\mathcal G$ of subset $X\subseteq V$ such that
$\{E_F(X)| X\in\mathcal{G}\}$ partitions $F$. Each $X\in\mathcal{G}$ induces a rigid subgraph
of $G(F)$ (see \cite{ChGS14} or the proof of Lemma 2.4 in \cite{JaJo05}).
By definition, a graph $G$ is rigid if and only if the rank of $\mathcal{R}(G)$ is $2|V(G)|-3$.

\begin{remark}[\cite{Gu17}]
\label{ransub}
Let $\mathcal G$ be a collection that realizes the minimum of the right side of (\ref{rankrm}), 
and $\mathcal Y\subseteq \mathcal G$. Then 
$r_{\mathcal R} \left(\cup_{X\in\mathcal{Y}} E_F(X)\right) = \sum_{X\in\mathcal{Y}} (2|X|-3)$.
\end{remark}

As in \cite{ChGS14}, $\mathcal{N}_{k,l}(G)$ is the matroid on ground set $E$ obtained by
taking matroid union of $k$ copies of the rigidity matroids $\mathcal{R}(G)$ and $l$ copies
of circuit matroids $\mathcal{M}(G)$. By a theorem of Edmonds on the rank 
of matroid union \cite{Edmo68}, the rank of $\mathcal{N}_{k,l}(G)$ is

\begin{equation}
\label{rankKL}
r_{k,l}(E)=\min_{F\subseteq E} \left\{kr_{\mathcal R}(F) + lr_{\mathcal M}(F) + |E-F| \right\}.
\end{equation}
Thus $r_{k,l}(E)\le  kr_{\mathcal R}(E) + lr_{\mathcal M}(E)= k(2n-3) + l(n-1)$.

\section{Proof of the packing theorem}
In this section, we prove Theorem \ref{parthm} and Corollary~\ref{parcor}.

\begin{proof}[\bf Proof of Theorem~\ref{parthm}]
It suffices to show that the rank of $\mathcal{N}_{k,l}(G)$ is
\[
r_{k,l}(E) = k(2n-3) + l(n-1).
\]
Choose $F\subseteq E$ to be a set with smallest size that minimizes the right side of 
(\ref{rankKL}), then 
\begin{equation}
\label{asrank}
r_{k,l}(E)= kr_{\mathcal R}(F) + lr_{\mathcal M}(F) + |E-F|.
\end{equation}
By (\ref{rankrm}), there exists a collection $\mathcal{X}$ of subset $X\subseteq V$
such that $\{E_F(X)| X\in\mathcal{X}\}$ partitions $F$ and
\begin{equation}
\label{rsrank}
r_{\mathcal R}(F)=\sum_{X\in\mathcal{X}} (2|X|-3).
\end{equation}

\noindent
{\bf Claim 1.} For each $X\in\mathcal{X}$, $|X|\ge 3$. \\
{\bf Proof of Claim 1.}  If not, then let $\mathcal{X}'$ denote
the collection of $X\in\mathcal{X}$ with $|X|=2$. Then
$r_{\mathcal R}(F)=\sum_{X\in\mathcal{X-X'}} (2|X|-3)+\sum_{X\in\mathcal{X'}} (2|X|-3)
=\sum_{X\in\mathcal{X-X'}} (2|X|-3)+|\mathcal{X'}|$. Let $H\subset F$ be the set of edges 
by deleting all edges induced by each $X$ with $|X|=2$. Then $\mathcal{X-X'}$ is the collection
of $X\subseteq V$ that partition $H$. By (\ref{rankrm}), 
$r_{\mathcal R}(H)\le \sum_{X\in\mathcal{X-X'}} (2|X|-3)$. 
As $G$ is simple, $|F-H|\le |\mathcal{X'}|$. Thus
$kr_{\mathcal R}(H) + lr_{\mathcal M}(H) + |E-H|\le k\sum_{X\in\mathcal{X-X'}} (2|X|-3)
+ lr_{\mathcal M}(F) + |E-F| + |F-H|\le 
k\sum_{X\in\mathcal{X-X'}} (2|X|-3)+ lr_{\mathcal M}(F) + |E-F| + |\mathcal{X'}|
\le kr_{\mathcal R}(F) + lr_{\mathcal M}(F) + |E-F|$, which is contrary to the minimality of $F$.
This completes the proof of the claim.\\

\noindent
{\bf Claim 2.} For every $\mathcal{Y}\subseteq \mathcal{X}$, there is a vertex
that is contained in a single element of $\mathcal{Y}$.\\
{\bf Proof of Claim 2.} 
If not, then every vertex is contained in at least two elements of $\mathcal{Y}$.
Let $n_\mathcal{Y}$ be the number of vertices in all elements of $\mathcal{Y}$. Then
$\sum_{X\in\mathcal{Y}}|X|\ge 2 n_\mathcal{Y}$. By Remark~\ref{ransub} and Claim 1, we have
\begin{eqnarray*}
2 n_\mathcal{Y} -3 
&\ge &  r_{\mathcal R} \left(\cup_{X\in\mathcal{Y}} E_F(X)\right) = \sum_{X\in\mathcal{Y}} (2|X|-3)\\
& = & \sum_{X\in\mathcal{Y}} |X| + \sum_{X\in\mathcal{Y}} (|X|-3)\\
& \ge & 2 n_\mathcal{Y} +0,
\end{eqnarray*}
a contradiction. This proves the claim.\\

Let $|V(G[F])|=n_1$ and $n_2=n-n_1$. Then there are $n_2$ isolated vertices in $G(F)$.
For each $X\in\mathcal{X}$, define $X_B = X\cap (\cup_{X\neq Y\in\mathcal{X}}Y)$ and $X_I=X-X_B$.
Let $\mathcal{I_X}=\{X\in\mathcal{X}: X_I\neq\emptyset\}$.
As each $X\in\mathcal{X}$ induces a connected subgraph of $G(F)$, it is not hard to see
\begin{equation}
\label{Fcomp}
c(F)\le |\mathcal{I_X}| + n_2.
\end{equation}
Proof of (\ref{Fcomp}).
Let $H$ be any connected component of $G(F)$ that is not an isolated vertex.
This $H$ is called a nontrivial component.
Each $X\in\mathcal{X}$ induces a connected subgraph of $G(F)$
and thus $H$ actually is a subgraph of $G(F)$ induced by some elements $X$'s of $\mathcal{X}$. 
Let $\mathcal{Y}$ be the collection of these $X$'s, and thus  $\mathcal{Y}\subseteq \mathcal{X}$. 
By Claim 2, there is a vertex $v$ in $V(H)$ that is contained in a single element of $\mathcal{Y}$. 
By definition, $v\in X_I$ and thus $X_I\neq\emptyset$. 
This shows that every nontrivial component of $G(F)$ contains an $X$ such that $X_I\neq\emptyset$.
Hence $G(F)$ contains at most $|\mathcal{I_X}|$ components that are not isolated vertices, 
which implies that $c(F)\le |\mathcal{I_X}| + n_2$ and completes the proof of (\ref{Fcomp}).
\\

Since $\mathcal{X}$ covers $F$ and thus covers all vertices of $G[F]$, each vertex of $X_B$
lies in at least two different $X\in\mathcal{X}$ and each $X_I$ is in a single $X$,
we have $\sum_{X\in\mathcal{X}}|X_B| + 2\sum_{X\in\mathcal{I_X}}|X_I|\ge 2n_1$, which implies
\begin{equation}
\label{covnum}
\sum_{X\in\mathcal{X}}|X| + \sum_{X\in\mathcal{I_X}}|X_I|\ge 2n_1.
\end{equation}

Now we will use the partition condition to show a lower bound of $|E-F|$.
Let $Z=\cup_{X\in\mathcal{X}} X_B$. Then $\{X_I: X\in\mathcal{I_X}\}$ together with all isolated
vertices of $G(F)$ form a partition $\pi$ of $G-Z$ with at least $n_2$ trivial parts and 
$|\pi|=|\mathcal{I_X}|+n_2$. Without loss of generality, we may assume there are exactly $n_2$
trivial parts (since this is the worst case for $e_{G-Z}(\pi)$).
Possibly there are  edges between trivial parts and $Z$. These
edges belong to $E-F$. Let $b$ be the number of those edges. Thus $n_Z(\pi)=b$, and we have
\begin{eqnarray}
\label{EminusF}
|E-F| \nonumber
&\ge & e_{G-Z}(\pi) + b \ge (3k+l)(|\pi|-1)- kn_2 - k|Z|(|\pi|-n_2) - n_Z(\pi) + b\\
& = & k\sum_{X\in \mathcal{I_X}}(3-|X_B|)+2kn_2 -3k + l(|\mathcal{I_X}| + n_2 -1) 
\end{eqnarray}

By (\ref{asrank}), (\ref{rsrank}), (\ref{Fcomp}), (\ref{covnum}), (\ref{EminusF}) and Claim 1,
\begin{eqnarray*}
r_{k,l}(E)
& =  & k\sum_{X\in\mathcal{X}} (2|X|-3) + l(n-c(F)) + |E-F|\\
&\ge & k(\sum_{X\in\mathcal{X}}|X| + \sum_{X\in\mathcal{I_X}} (|X|-3)) + l(n-c(F))\\
&    & + k\sum_{X\in \mathcal{I_X}}(3-|X_B|)+2kn_2 -3k + l(|\mathcal{I_X}| + n_2 -1)\\
& =  & k(\sum_{X\in\mathcal{X}}|X| + \sum_{X\in\mathcal{I_X}}|X_I| +2n_2 -3)
       + l(n-c(F)+ |\mathcal{I_X}| + n_2 -1)\\
&\ge & k(2n_1 + 2n_2 -3) + l(n-1) + l(|\mathcal{I_X}| + n_2 -c(F))\\
&\ge & k(2n-3) + l(n-1).
\end{eqnarray*}
As $r_{k,l}(E)\le k(2n-3) + l(n-1)$, it turns out that $r_{k,l}(E)= k(2n-3) + l(n-1)$.
\end{proof}

\vspace{.1cm}

\begin{proof}[\bf Proof of Corollary~\ref{parcor}]
By Theorem \ref{parthm},
it suffices to show that for  any partition $\pi$ of $V(G-Z)$ with $n_0$ trivial parts, 
$e_{G-Z} (\pi)\ge (3k+l)(|\pi|-1) - k n_0 - k|Z|(\pi -n_0) - n_Z(\pi)$, for every $Z\subset V(G)$.
Let $u_j$ be a trivial part (single vertex), then $d(u_j)\ge 4k+2l -|Z_j|$, where $Z_i$ is 
the set of vertices in $Z$ that are adjacent to $u_i$. 
Let $V_i$ be a nontrivial part in the partition. If $V_i$ induces at least one edge in $G-Z$, then 
$d(V_i)\ge 6k+2l - 2k|Z|$ by the essential edge connectivity. 
If $V_i$ is a independent set of $G-Z$, then
$d(V_i)\ge |V_i|(4k+2l -|Z|)\ge 2(4k+2l -|Z|) \ge 6k+2l - 2k|Z|$. 
Thus $d(V_i)\ge 6k+2l - 2k|Z|$. Hence
\begin{eqnarray*}
e_{G-Z} (\pi) 
&\ge & \frac{1}{2}\sum_1^{|\pi|-n_0} (6k+2l - 2k|Z|) + \frac{1}{2}\sum_1^{n_0} (4k+2l -|Z_j|)\\
&\ge & (3k+l)|\pi| - k n_0 - k|Z|(\pi -n_0) - \frac{1}{2}\sum_1^{n_0}|Z_j| \\
&\ge & (3k+l)(|\pi|-1) - k n_0 - k|Z|(\pi -n_0) - n_Z(\pi),
\end{eqnarray*}
which completes the proof.
\end{proof}

\section{Necessary conditions}

Theorem~\ref{necondi} presents a necessary partition condition. 
As corollaries, we obtain some properties of rigid graphs. 
\begin{theorem}
\label{necondi}
Let $k\ge 0$ and $l\ge 0$ be integers.
If a graph $G$ contains edge-disjoint $k$ spanning rigid subgraphs and $l$ spanning trees,
then for any partition $\pi$ of $V(G)$ with $n_0$ trivial parts, 
$e_G (\pi)\ge (3k+l)(|\pi|-1) - k n_0$.
\end{theorem}
\begin{proof}[\bf Proof of Theorem \ref{necondi}]
Let $S$ be a spanning subgraph of $G$ that consists of edge-disjoint $k$ spanning minimally 
rigid subgraphs and $l$ spanning trees. By definition, $|E(S)|=k(2n-3)+l(n-1)$, where $n=|V(G)|$.
Let $\pi=\{V_1,V_2,\cdots,V_t,\cdots,V_{t+n_0}\}$ be a partition of $V(G)$ such that
$V_i$ is nontrivial for $1\le i\le t$ and trivial for $t+1\le i\le t+n_0$.
Thus $\sum_{1\le i\le t}|V_i| = n - n_0$ and $t=|\pi|-n_0$.
For $1\le i\le t$, $|E(S[V_i])|)\le k(2|V_i|-3)+l(|V_i|-1)$. Then 
\begin{eqnarray*}
e_G(\pi)
&\ge & e_S(\pi)=|E(S)|-\sum_{1\le i\le t}|E(S[V_i])| \\
&\ge & k(2n-3)+l(n-1)-\sum_{1\le i\le t}(k(2|V_i|-3)+l(|V_i|-1))\\
& =  & k(2n-3)+l(n-1)-k(2n-2n_0-3t)-l(n-n_0-t)\\
& =  & k(2n_0 + 3t -3) + l(n_0 +t -1)\\
& =  & k(2n_0 + 3|\pi|-3n_0 -3) + l(n_0 +|\pi| -n_0 -1)\\
& =  & (3k+l)(|\pi|-1) - k n_0. 
\end{eqnarray*}
\end{proof}

\begin{corollary}
Every rigid graph is essentially $3$-edge-connected.
\end{corollary}
\begin{proof}[\bf Proof]
The proof follows by Theorem~\ref{necondi} when $k=1, l=0, n_0=0$ and $|\pi|=2$.
\end{proof}

\begin{corollary}
Every rigid graph $G$ with $|E(G)|\ge 2(|V(G)|-1)$ has $2$ edge-disjoint spanning trees.
\end{corollary}
\begin{proof}[\bf Proof]
By Theorem~\ref{NaTu}, it suffices to show that for any partition $\pi$ of $V(G)$,
$e_G (\pi)\ge 2(|\pi|-1)$. If $|\pi|=|V(G)|$, then $e_G (\pi)=|E(G)|\ge 2(|V(G)|-1)=
2(|\pi|-1)$. Thus we may assume that  $|\pi| < |V(G)|$. Then $n_0\le |\pi|-1$,
where  $n_0$ is the number of trivial parts of $\pi$.
By Theorem~\ref{necondi}, $e_G (\pi)\ge 3(|\pi|-1) - n_0\ge 2(|\pi|-1)$.
\end{proof}

Notice that there is a gap between the sufficient condition (Theorem~\ref{parthm}) and 
necessary condition (Theorem~\ref{necondi}).
The spanning tree packing theorem by Nash-Williams~\cite{Nash61} and Tutte~\cite{Tutte61}
suggests that there might be a partition characterization of packing spanning trees and spanning 
rigid subgraphs. Thus we pose the following problem.
\begin{problem}
Find a partition condition to characterize graphs with edge-disjoint $k$ spanning rigid 
subgraphs and $l$ spanning trees.
\end{problem}

\section{Sparse subgraph covering}
\label{boundsect}

As a dual problem of packing, subgraph covering also attracts much attention.
Nash-Williams published the following result, characterizing
graphs that can be decomposed to $k$ forests.

\begin{theorem}[Nash-Williams \cite{Nash64}]
\label{nashcov}
Let $k\ge 0$ be an integer.
A connected graph $G$ can be decomposed to $k$ forests if and only if
for any nonempty subset $X\subseteq V(G)$, $i_G(X)\le k(|X|-1)$.
\end{theorem}

Edmonds extended the above forest decomposition theorem to matroids \cite{Edmo65}.
Apply Edmonds theorem to rigidity matroids, we have the following sparse subgraph 
decomposition theorem. Alternatively, a short proof is given here.

\begin{theorem}
\label{rigidcover}
A connected graph $G$ can be decomposed into $k$ sparse subgraphs if and 
only if for any subset $X\subseteq V(G)$ with $|X|\ge 2$, $i_G(X)\le k(2|X| - 3)$.
\end{theorem}
\begin{proof}[\bf Proof of Theorem \ref{rigidcover}]
Suppose that $G$ decomposes into $k$ spanning sparse subgraphs.
By the definition of sparse graphs, for any subset $X\subseteq V(G)$ with $|X|\ge 2$,
$i_G(X)\le k(2|X|-3)$, which proves the necessity.

To prove the sufficiency, assume that for any subset $X\subseteq V(G)$ with $|X|\ge 2$,
$i_G(X)\le k(2|X| - 3)$. It suffices to show the rank of $\mathcal{N}_{k,0}(G)$,
$r_{k,0}(E) \ge |E|$.

Let $F\subseteq E$ be a set that minimizes the right side of (\ref{rankKL}) when $l=0$, then 
\[
r_{k,0}(E)= kr_{\mathcal R}(F) + |E-F|.
\]
By (\ref{rankrm}), there exists a collection $\mathcal{G}$ of subset $X\subseteq V$
such that $\{E_F(X)| X\in\mathcal{G}\}$ partitions $F$ and
\[
r_{\mathcal R}(F)=\sum_{X\in\mathcal{G}} (2|X|-3).
\]
Then $|F|=\sum_{X\in\mathcal{G}}i_F(X)\le \sum_{X\in\mathcal{G}} k(2|X| - 3)
=k\sum_{X\in\mathcal{G}} (2|X|-3)=k r_{\mathcal R}(F)$.
Thus $|E|=|F|+|E-F|\le kr_{\mathcal R}(F) + |E-F| = r_{k,0}(E)$,
which implies that $E$ is an independent set of $\mathcal{N}_{k,0}(G)$. This completes the proof.
\end{proof}

For a graph $G$, the {\bf fractional arboricity} $\gamma(G)$ of $G$ is defined as
\[
\gamma(G)=\max_{X\subseteq V(G)} \frac{i_G(X)}{|X|-1},
\]
whenever the denominate is nonzero.
This notation was introduced by Payan \cite{Paya86} and was generalized to matroids
by Catlin et al. \cite{CGHL92}. The well-known theorem (Theorem~\ref{nashcov}) of 
Nash-Williams on forest covering indicates that $G$ decomposes to $\lceil\gamma(G)\rceil$
forests. When $\gamma(G)=k+\epsilon$ with $0<\epsilon<1$, Nash-williams's theorem tells us
$G$ decomposes to $k+1$ forests but does not give any information on different $\epsilon$
values. Towards this observation, Montassier et al. \cite{MORZ12} posed the following
Nine Dragon Tree Conjecture stating that the maximum degree of one of the forests should
be bounded by a function of $\epsilon$. They also have a Weaker NDT Conjecture if the
degree bounded forest is replaced by a degree bounded subgraph.

\begin{conjecture}[NDT Conjecture \cite{MORZ12}]
If $\gamma(G)=k+\epsilon$ with $0<\epsilon<1$, then $G$ decomposes into $k+1$ forests,
one of which has maximum degree at most $\lceil\frac{(k+1)\epsilon}{1-\epsilon}\rceil$.
\end{conjecture}

\begin{conjecture}[Weaker NDT Conjecture \cite{MORZ12}]
If $\gamma(G)=k+\epsilon$ with $0<\epsilon<1$, then $G$ decomposes into $k$ forests and
a subgraph with maximum degree at most $\lceil\frac{(k+1)\epsilon}{1-\epsilon}\rceil$.
\end{conjecture}

The NDT Conjecture was settled by Jiang and Yang \cite{JiYa16}.

Notice that if $\epsilon$ is large (close to $1$), the Weaker NDT Conjecture states that
the maximum degree of the subgraph is bounded by a very large number, which seems give no
information about the subgraph. Motivated by this observation, we are interested in the upper
bound of the maximum degree of the subgraph in general.

For a graph $G$, $\gamma_2 (G)$ of $G$ is defined as
\[
\gamma_2(G)=\max_{X\subseteq V(G)} \frac{i_G(X)}{2|X|-3},
\]
whenever $|X|\ge 2$. We have the following result.

\begin{theorem}
\label{boundedcov}
Let $k,l\ge 0$ be integers with $k+1\le l\le 2k+2$.
If $\gamma_2 (G)\le k+1$, then $G$ decomposes into $l$ forests and
$2k+2 -l$ subgraphs with maximum degree at most $(2|V(G)|-5)/3$.
\end{theorem}

We need the following lemmas to prove Theorem~\ref{boundedcov}.

\begin{lemma}[\cite{KKWWZ14}]
\label{kwzthm}
For $d\ge k+1$, if $(k+1)(k+d)|X| - (k+d+1)i_G(X)-k^2\ge 0$ for every nonempty subset $X\subseteq V(G)$, 
then $G$ decomposes into $k$ forests and a subgraph with maximum degree at most $d$.
\end{lemma}

\begin{lemma}
\label{spadec}
Any sparse graph $G$ decomposes into a forest and a subgraph with maximum degree at most $(2|V(G)|-5)/3$.
\end{lemma}
\begin{proof}[\bf Proof]
It suffices to show that any sparse graph satisfies the condition $(k+1)(k+d)|X| - (k+d+1)i_G(X)-k^2\ge 0$
when $k=1$ and $d=(2|V(G)|-5)/3$ in Lemma \ref{kwzthm}. By definition, for any sparse graph $G$ and
$X\subseteq V(G)$ with $|X|\ge 2$, $i_G(X)\le 2|X|-3$. Then
$2(1+d)|X| - (d+2)i_G(X)-1\ge 2(1+d)|X| - (d+2)(2|X|-3)-1 = 3d -2|X| +5\ge 0$, completing the proof.
\end{proof}

\begin{lemma}
\label{spafore}
Any sparse graph decomposes into two forests.
\end{lemma}
\begin{proof}[\bf Proof]
The lemma follows easily from Theorem~\ref{nashcov} and from the definition of sparse graphs.
\end{proof}

\begin{proof}[\bf Proof of Theorem~\ref{boundedcov}]
As $\gamma_2(G)\le k+1$, we have $i_G(X)\le (k+1)(2|X|-3)$.
By Theorem~\ref{rigidcover}, $G$ decomposes into $k+1$ sparse subgraphs.
By Lemma~\ref{spafore}, $l-k-1$ sparse subgraphs decompose into $2l-2k-2$ forests.
By Lemma~\ref{spadec}, the other $2k+2-l$ sparse subgraphs decompose into $2k+2-l$ forests and
$2k+2-l$ subgraphs with maximum degree at most $(2|V(G)|-5)/3$. Thus $G$ can decompose 
into $l$ forests and $2k+2-l$ subgraphs with maximum degree at most $(2|V(G)|-5)/3$.
\end{proof}

As an analogue of Nine Dragon Tree problem, we pose the following sparse subgraph
covering problem.

\begin{problem}
Find a minimum integer $f(k,\epsilon)$ such that if $\gamma_2(G)=k+\epsilon$ 
with $0<\epsilon<1$, then $G$ decomposes into $k+1$ sparse subgraphs,
one of which has maximum degree at most $f(k,\epsilon)$.
\end{problem}

\section{Acknowledgment}
The author would like to thank Viet Hang Nguyen for pointing out a mistake in a previous version.
The author is partially supported by a grant from the Simons Foundation (No. 522728, XG).

\end{document}